\documentclass[12pt]{amsart}
\usepackage[utf8]{inputenc}
\usepackage[active]{srcltx}
\usepackage{tabularx}
\usepackage{mathrsfs}
\usepackage{array,booktabs,arydshln,xcolor}
\usepackage{fullpage}
\usepackage{amsfonts,enumerate}
\usepackage[mathscr]{euscript}
\usepackage{epsfig}
\usepackage{latexsym}
\usepackage{leftindex}
\usepackage[all]{xy}
\usepackage{amssymb}
\usepackage{amsthm}
\usepackage{float}
\usepackage{amsmath}
\usepackage{amsxtra}
\usepackage{xcolor}
\usepackage[colorlinks]{hyperref}
\usepackage{fancyvrb}
\usepackage{tikz}
\hypersetup{citecolor=blue}
\usepackage{caption} 
.tfm
.tfm

\theoremstyle{plain}
\newtheorem{prop}{Proposition}[section]
\newtheorem{thm}[prop]{Theorem}
\newtheorem*{thm*}{Theorem}
\newtheorem{coro}[prop]{Corollary}

\newtheorem{lemma}[prop]{Lemma}

\newtheorem{ass}{Assumption}
\newtheorem*{lemma*}{Lemma}
\newtheorem*{prop*}{Proposition}

\theoremstyle{definition}

\theoremstyle{remark}

\newtheorem{remark}{Remark}

\numberwithin{table}{section}

\DeclareMathOperator{\Gal}{Gal}

\newcommand{\GL}{{\rm GL}}

\makeatletter
\newcommand{\verbatimfont}[1]{\renewcommand{\verbatim@font}{\ttfamily#1}}
\makeatother

\def\ZZ{\mathbb Z}

\def\FF{\mathbb F}
\def\QQ{\mathbb Q}

\def\ff{\mathfrak f}
\def\pp{\mathfrak p}

\def\qq{\mathfrak q}

\def\<#1>{{\left\langle{#1}\right\rangle}}

\def\id#1{{\mathfrak{#1}}}      

\DeclareMathOperator{\norm}{{\mathscr N}}

\usetikzlibrary{calc}

\begin{document}
	
	\title{On Fermat's Last Theorem over the $\mathbb{Z}_3$-extension of $\mathbb{Q}$ and other fields}
	
	\author{Luis Dieulefait}
	\address{Universitat de Barcelona, G. V. de les Corts Catalan es 585, 08007 Barcelona, Spain} 
	\email{ldieulefait@ub.edu}
	\author{Franco Golfieri Madriaga}
	\address{Center for Research and Development in Mathematics and Applications (CIDMA),
		Department of Mathematics, University of Aveiro, 3810-193 Aveiro, Portugal}
	\email{francogolfieri@ua.pt}


		\begin{abstract}

            The main result of the present article is a proof that, over certain totally real number fields $K$, the Fermat equation $x^p+y^p+z^p=0$ has no non-trivial primitive solutions in $\mathcal{O}_K$ for sufficiently large prime exponents $p$ with $p \equiv 2 \pmod{3}$. A particular case of these fields are the maximal real subfields of the cyclotomic extensions $\QQ(\zeta_{3^n})$ for every $n$. Our strategy consists in combining the modular method with a generalization of an arithmetic result of Pomey to these fields.
            
		\end{abstract}
	
	\maketitle

	\section{Introduction}

	
   
    Let $K$ be a number field and let denote $\mathcal{O}_K$ its ring of integers. Let us consider the Fermat equation 
    \begin{equation}
    \label{fermat:eq}
    x^p+y^p=z^p, \ \ x,y,z \in \mathcal{O}_K.
    \end{equation}
    We say that the \textbf{asymptotic Fermat's Last Theorem} holds for $K$ if there exists a constant $C_K$, depending only on $K$, such that for all primes $p>C_K$, the only solutions $(a,b,c) \in \mathcal{O}_K^3$ to the equation \eqref{fermat:eq} are the trivial ones, i.e. the solutions satisfying $abc=0$.

Throughout this article we say that a solution $(a,b,c) \in \mathcal{O}_K^3$ is \textit{non-trivial} if $abc \neq 0$, and \textit{primitive} if $a\mathcal{O}_K + b\mathcal{O}_K + c\mathcal{O}_K = \mathcal{O}_K$; for a solution of the Fermat equation this is equivalent to $a,b,c$ being pairwise coprime. \textbf{All the results of this article concern only primitive solutions}: what we prove is the non-existence of non-trivial \emph{primitive} solutions. When $\mathcal{O}_K$ is a principal ideal domain every non-trivial solution is a multiple of a non-trivial primitive one, so in that case our results rule out all non-trivial solutions; for general $K$, however, a non-trivial solution need not be a multiple of a primitive one, and our results say nothing about such solutions.

In \cite{freitas2015asymptotic}, Freitas and Siksek proved the asymptotic Fermat's Last Theorem for a set of real quadratic number fields of density $5/6$ (and, more generally, gave a criterion for totally real fields), using the modularity of elliptic curves over real quadratic fields proved in \cite{freitas2015elliptic}. In their work, they proved it by studying the $S$-unit solutions to the unit equation. 
    
    Recently, motivated by the work of Freitas, Siksek, and Le Hung, there has been a great advance in the study of asymptotic Fermat's Last Theorem. In this article, we present alternative conditions for our fields $K$, independent from the conditions of the $S$-units stated by Freitas and Siksek (see \cite[Theorem 3]{freitas2015asymptotic}).
  
  In 1931, Massouti\'e \cite{massoutie1931dernier} and Pomey \cite{pomey1931nouvelles} proved the following result:
    
    \begin{thm*}
    	\label{thm:pomey}
    If $p$ is an odd prime such that $p \equiv 2 \pmod 3$ and $x,y,z$ are non-zero integers such that $x^p+y^p+z^p=0$, then $3$ divides $x,y,$ or $z$.
    \end{thm*}

    In their proof, they use the following elementary result.
    \begin{lemma*}
    Let $p$ be an odd rational prime and $a,b$ rational coprime integers. If $p|a^2+3b^2$, then $p$ is of the form $x^2+3y^2$.
    \end{lemma*}
    Note that the oddness assumption is necessary: $2$ divides $1^2+3\cdot 1^2=4$ but $2$ is not of the form $x^2+3y^2$. This Lemma is the main obstacle to generalizing Pomey's result for general number fields. Motivated by Pomey's proof, we can generalize the above lemma by adding an extra hypothesis to the class field of $K$ and the ramification of $3$. A delicate point in the generalization is that the ring of integers of $K(\sqrt{-3})$ is in general not known (and it is in general strictly larger than $\mathcal{O}_K[\sqrt{-3}]$); our argument is designed so that this knowledge is not needed.
    In order to state the result, we need to consider the following notation. Let $K$ be a totally real number field and let $q$ be an odd prime. Let us denote,
    \begin{equation}\label{eq:first_conditions}
    \begin{split}
    S_q &= \{\qq : \qq \ \text{is a prime of } K \ \text{dividing} \ q\},\\
    T_q = \{\qq \in &S_q : f(\qq|q) = 1\}, \ \ V_q= \{\qq \in S_q : 2 \nmid e(\qq|q) \},
    \end{split}
    \end{equation}
    where $f(\qq|q)$ denotes the residual degree of $\qq$. 
    
    We also make the following assumption.  
\begin{ass}
Keeping the previous notation we say that $K$ satisfies Assumption~\ref{assumption} if
\begin{enumerate}
\item $T_3 \cap V_3 \neq \emptyset$, i.e. there is a prime $\pp_3$ of $K$ above $3$ with $f(\pp_3|3)=1$ and $e(\pp_3|3)$ odd,
\item $h(K(\sqrt{-3}))$ is odd.
\end{enumerate}
\label{assumption}
\end{ass}

\begin{remark}
\label{rem:assumption}
In a previous version of this article, condition (1) was replaced by the two weaker conditions $T_3 \neq \emptyset$ and $V_3 \neq \emptyset$. The present (slightly stronger) condition is needed in the corrected proof of Theorem~\ref{pomey_case}, where we use a single prime $\pp_3$ above $3$ which has residue degree $1$ and odd ramification index at the same time. As we will see in Section~\ref{ejemplos}, this change does not affect the examples considered in this article (see Remark~\ref{options} and Table~\ref{table:1}).
\end{remark}

We will also use the following notation: for a prime $p$ we say that $p$ is \textit{admissible for $K$} if $p \nmid \norm(\qq)-1$ for every prime $\qq$ of $K$ dividing $6$. Observe that every prime $p > \max\{\norm(\qq) : \qq \mid 6\}$ is admissible for $K$.

Then, in section \ref{3nmidabc} we are able to prove the following  theorem.

\begin{thm}
\label{divpomey}
Let $K$ be a totally real field satisfying Assumption~\ref{assumption} and let $L = K(\sqrt{-3})$ and $t:=h(L)$. Let $a,b,d \in \mathcal{O}_K$ with $d$ totally positive, $d$ coprime to $6$, $d\mathcal{O}_K + a\mathcal{O}_K + b\mathcal{O}_K = \mathcal{O}_K$, and $d|a^2+3b^2$. Then $d^t = \norm_{L/K}(\beta)$ for some $\beta \in \mathcal{O}_L$. In particular, $d^t = x^2+3y^2$ with $x,y \in K$ (not necessarily integral).
\end{thm}

Motivated by Pomey's proof, Theorem \ref{divpomey} enables us to prove the following generalization of Pomey's theorem for fields satisfying Assumption~\ref{assumption}.

\begin{thm}
\label{pomey_case}
Let $K$ be a totally real field satisfying Assumption~\ref{assumption} and let $\pp_3$ be a prime of $K$ above $3$ with $f(\pp_3|3)=1$ and $e(\pp_3|3)$ odd. Let $p$ be an odd prime such that $p \equiv 2 \pmod 3$ and $p$ is admissible for $K$. If $(x,y,z) \in \mathcal{O}_K^3$ is a non-trivial primitive solution to $x^p+y^p+z^p=0$, then $\pp_3$ divides $xyz$.
\end{thm}

Observe that for these particular fields, we rule out the non-trivial primitive solutions $(a,b,c)$ to the equation \eqref{fermat:eq} with $\pp_3 \nmid abc$ (under the assumption that $p \equiv 2 \pmod{3}$ and $p$ is admissible for $K$). In what follows, when we write $3 \mid abc$ for $a,b,c \in \mathcal{O}_K$ we mean that some prime ideal of $K$ above $3$ divides $abc$; in particular, the conclusion of Theorem~\ref{pomey_case} implies that $3 \mid xyz$ in this sense.

To deal with the case $3|abc$ we use modularity results over these number fields and lowering the level results at the prime $3$. To guarantee  modularity, we need the exponent $p$ to be greater than a certain bound that will only depend on $K$.

\begin{thm}
    \label{3midabc}
    Let $K$ be a totally real field. There exists a bound $C_K$, only depending on $K$, such that if $p>C_K$, there are no non-trivial primitive solutions $(x,y,z) \in \mathcal{O}_K^3$ on $K$ to the equation $x^p+y^p+z^p=0$ satisfying $3|xyz$.
    \end{thm}

Finally, Theorem \ref{pomey_case} and Theorem $\ref{3midabc}$ imply our main result.

   \begin{thm}
    \label{thm:main}
  	Let $K$ be a totally real field satisfying Assumption~\ref{assumption}. Then there exists a bound $C_K$ such that if $p \equiv 2 \pmod 3$ and $p>C_K$, there are no non-trivial primitive solutions $(x,y,z) \in \mathcal{O}_K^3$ to the equation $x^p+y^p+z^p=0$.
   \end{thm}

   \begin{remark}
   We cannot state the result as \textit{``the asymptotic Fermat's Last Theorem holds over $K$''} for two reasons. First, the result applies only to primes $p$ with $p\equiv 2 \pmod{3}$; this restriction comes from Theorem \ref{pomey_case}. Second, the result only concerns primitive solutions: unless every non-trivial solution is a multiple of a primitive one (for instance, when $\mathcal{O}_K$ is a principal ideal domain), it does not exclude the existence of non-trivial non-primitive solutions.
   \end{remark}
   
   Observe that in this result, we do not impose conditions on the S-units solutions over $K$, as in \cite[Theorem 3]{freitas2015asymptotic}. The only conditions we require are about the ramification of certain primes and the class number of $K(\sqrt{-3})$. A similar result, depending only on the class number of $K$ and the ramification of certain primes, is proved, using other techniques, for the equations $x^2+dy^6=z^p$ and $x^4+dy^2=z^p$ in \cite{gpvillagra}.
   
   This presents an advantage since, in general, it is difficult to compute the $S$-unit solutions to the unit equation on $K$. On the other hand, assuming the generalized Riemann hypothesis, it is fast to compute the class number of $K(\sqrt{-3})$. However, it does not present a full alternative to the $S$-unit approach since we only cover half of the primes, i.e. the odd primes such that $p \equiv 2 \pmod3$.

   We can then prove this partial asymptotic result for some particular number fields. One field of interest is     $\QQ(\zeta_{3^n})^+$ ,  the maximal totally real subfield of the cyclotomic extension $\QQ(\zeta_{3^n})$, for which no general result is proven. We are able to prove the following result. 

   \begin{thm}
   \label{cubic_cyclo}
   Let $n$ be a positive integer. Then there exists a bound $C$, depending only on $n$, such that if $p$ is an odd prime with $p \equiv 2 \pmod{3}$ and $p>C$, then there are no non-trivial primitive solutions on $\QQ(\zeta_{3^n})^+$ to the equation $x^p+y^p+z^p=0$.
   \end{thm}

  The article is organized as follows. In section \ref{3|abc} we analyze the case $3|abc$ , where $(a,b,c)$ is a putative solution to the equation \ref{fermat:eq}, and prove Theorem \ref{3midabc}. In order to do so, we use modularity results over number fields, so that we can associate a space of Hilbert modular forms to our Frey curve and then, using that $3|abc$, finally arrive at a contradiction. This is the part where the bound $C_K$ appears. This bound will only depend on the field $K$. In section \ref{3nmidabc} we analyze the case $3\nmid abc$, and prove Theorem \ref{divpomey} and Theorem $\ref{pomey_case}$. Finally, in section \ref{ejemplos} we present some examples of fields that satisfy Assumption~\ref{assumption} and prove Theorem \ref{cubic_cyclo}.

    	
    	
    	
    	

\subsection*{Acknowledgements} We thank Ariel Pacetti and Nuno Freitas for many useful comments and helpful suggestions. We are especially grateful to Alain Kraus for pointing out an error in a previous version of this paper; correcting it required a slight change in the conditions imposed on the number field.

\subsection*{Funding} Luis Dieulefait was supported by the Spanish Ministry of Science, Innovation and Universities through project PID2022-136944NB-I00. Franco Golfieri Madriaga was supported by a grant from the Portuguese Foundation for Science and Technology (UI\slash BD\slash 152573\slash 2022).

    \section{The case $3|xyz$} \label{3|abc}
    
    Let $K$ be a totally real number field, and let $(x,y,z)= (a,b,c) \in \mathcal{O}_K^3$ be a putative non-trivial primitive solution to the equation (\ref{fermat:eq}). Suppose $3|abc$, i.e. some prime ideal of $K$ above $3$ divides $abc$. As in the proof of the classical Fermat's Last Theorem, we associate to this solution  the Hellegouarch-Frey elliptic curve
    \begin{equation}
    \label{frey-curve}
    E_{(a,b,c)}: y^2 = x(x-a^p)(x+b^p)
    \end{equation}
    defined over $K$. Let $\overline{\rho}_{E,p}$ denote its mod $p$ Galois representation. 

    Let us define 
    \begin{align*}
    S_K = \{\mathfrak{P} : \mathfrak{P} \ \text{is a prime of } K \ \text{above} \ 2\}.
    \end{align*} 
   
    \begin{lemma}
    For all primes $\qq \not\in S_K$, the model of $E_{(a,b,c)}$ given in \eqref{frey-curve} is minimal, semistable and satisfies $p | v_{\qq}(\Delta_{E_{(a,b,c)}})$. The representation $\Bar{\rho}_{E_{(a,b,c)},p}$ is odd, its determinant is the cyclotomic character and is finite flat at every prime $\pp$ of $K$ dividing $p$. Moreover, its conductor is
    \begin{equation}
    \label{conductor}
    \mathcal{N}_{E_{(a,b,c)}} = \prod_{\mathfrak{P} \in S_K} \mathfrak{P}^{r_{\mathfrak{P}}} \cdot \prod_{\substack{\qq \mid abc \\ \qq \notin S_K}} \qq,
    \end{equation}
    where $0 \leq r_{\mathfrak{P}} \leq 2+6 v_{\mathfrak{P}}(2)$. In particular, the ideal $\mathfrak{N}_p$ of Theorem~\ref{lowering} below is of the form $\mathfrak{N}_p = \prod_{\mathfrak{P} \in S_K} \mathfrak{P}^{r'_{\mathfrak{P}}}$ with $0 \leq r'_{\mathfrak{P}} \leq r_{\mathfrak{P}}$.
    \end{lemma}

    \begin{proof}
    The discriminant of $E_{(a,b,c)}$ is equal to $\Delta_{E_{(a,b,c)}} = 16 \cdot (abc)^{2p}$, then $p|v_{\qq}(\Delta_{E_{(a,b,c)}})$. From \cite[Lemma3.3]{freitas2015asymptotic} we obtain that the model of $E_{(a,b,c)}$ is minimal, semistable and that its conductor is the one expressed in \eqref{conductor}, where $0 \leq r_{\mathfrak{P}} \leq 2+6 v_{\mathfrak{P}}(2)$. The determinant and oddness is a well-known consequence of the theory of the Weil pairing on $E_{(a,b,c)}[p]$. Finally, it follows from \cite{Serre87} that $\Bar{\rho}_{E_{(a,b,c)},p}$ is unramified at every prime $\qq \notin S_K$ with $\qq \nmid p$ (at the primes $\qq \mid abc$ this uses that $p \mid v_{\qq}(\Delta_{E_{(a,b,c)}})$), and is finite flat at every $\pp$ dividing $p$; at the primes of $S_K$ it may be ramified. 
    \end{proof}

    \subsection{Modularity}
    
    Let $E$ be an elliptic curve defined over a totally real field $K$. We say that $E$ is modular if there exists an isomorphism of Galois representations
    \begin{align}
    \rho_{E,p} \sim \rho_{\ff, \Bar{\omega}},
    \end{align}
    where $\ff$ is a Hilbert cuspidal Hecke eigenform over $K$ of parallel weight $2$ with coefficient field $K_f$ and $\Bar{\omega}$ is a prime in $K_f$ above $p$.

    In (\cite[Theorem 5]{freitas2015elliptic}), Freitas, Le Hung, and Siksek proved the following result about the modularity of elliptic curves.
    
    \begin{thm}
    	Let $K$ be a totally real field. All but finitely many $\overline{K}$-isomorphism classes of elliptic curves over $K$ are modular.
    \end{thm}
    
    In particular, it implies the following corollary.
    
    \begin{coro}
    \label{modular}
    	Let $K$ be a totally real field. There exist some constant $A_K$ depending only on $K$, such that for any non-trivial solution $(a,b,c)$ of the Fermat equation (\ref{fermat:eq}) with prime exponent $p>A_K$, the Frey curve $E_{(a,b,c)}$ given by \eqref{frey-curve} is modular.
    \end{coro}
    
    \begin{proof}
    	See \cite[Corollary 2.1]{freitas2015asymptotic}.
    \end{proof}


    Therefore, if $p>A_K$, there exists a Hilbert cuspidal eigenform $\frak{f}$ over $K$ of parallel weight $2$ and level $\mathcal{N}_{E_{(a,b,c)}}$ such that
    \begin{align}
    \rho_{E,p} \sim \rho_{\ff, \Bar{\omega}},
    \end{align}

    \subsection{Irreducibility}

    To obtain irreducibility of the residual Galois representation $\bar{\rho}_{E_{(a,b,c)},p}$ we use the following result 

    \begin{thm} (\cite[Theorem 6]{anni2016modular})
    \label{irreducible}
    Let $K$ be a totally real field. There is an effectively computable constant $B_K$ such that for a prime $\ell > B_K$, and for an elliptic curve $E/K$ semistable at all $\lambda | \ell$, the mod $\ell$ representation $\Bar{\rho}_{E,\ell}:  G_K \rightarrow \GL_2(\FF_{\ell})$ is irreducible.
    \end{thm}

\begin{remark}
As it is pointed out in \cite{anni2016modular} this result is a restatement of \cite[Theorem 2]{freitas2015criteria}, replacing $K$ with its Galois closure.
\end{remark}
    
Then, since $E_{(a,b,c)}$ is semistable away from $2$ we are under the hypothesis of a lowering the level result. 

\subsection{Level lowering} 

We use the following result from Freitas and Siksek.

\begin{thm} (\cite[Theorem 7]{freitas2015asymptotic})
\label{lowering}
Let $K$ be a totally real field and $E/K$ an elliptic curve of conductor $\mathfrak{N}_E$. Let $p$ be a rational prime. For a prime ideal $\qq$ of $K$, denote by $\Delta_{\qq}$ the discriminant of a local minimal model of $E$ at $\qq$. Let

    \begin{align}
    \mathfrak{M}_p = \prod_{\substack{\qq || \mathfrak{N}_E \\
    p|v_{\qq}(\Delta_{\qq})}} \qq, \ \ \ \mathfrak{N}_p = \frac{\mathfrak{N}_E}{\mathfrak{M}_p}.
    \end{align}
    
Suppose the following conditions hold: 
\begin{itemize}
\item $p\geq 5$, the ramification index $e(\pp|p)<p-1$ for every prime $\pp \mid p$ of $K$, and $\QQ(\zeta_p)^+ \not\subseteq K$,
\item $E$ is modular,
\item $\Bar{\rho}_{E,p}$ is irreducible,
\item $E$ is semistable at all $\qq|p$,
\item $p|v_{\pp}(\Delta_{\pp})$ for all $\pp|p$.
\end{itemize}
Then, there is a Hilbert eigenform $\ff$ of parallel weight $2$ that is new at level $\mathfrak{N}_p$ and some prime $\Bar{\omega}$ of $K_f$ above $p$ such that $\Bar{\rho}_{E,p} \sim \Bar{\rho}_{f,\Bar{\omega}}$.
\end{thm}

Let us fix $K$ a totally real field. Let $D_K$ a bound such that if $p>D_K$, then $e(\pp|p)<p-1$ for every prime $\pp|p$ of $K$, and $\QQ(\zeta_p)^+ \not\subseteq K$. If $p>A_K$ then $E_{(a,b,c)}$ is modular by Corollary (\ref{modular}). If $p>B_K$ then $\Bar{\rho}_{E_{(a,b,c)},p}$ is irreducible by Theorem (\ref{irreducible}). Finally, since $E_{(a,b,c)}$ is semistable away from $2$, we conclude that $E_{(a,b,c)}$ satisfies all the hypothesis of Theorem (\ref{lowering}). Therefore there exists a Hilbert eigenform $\ff$ of parallel weight $2$ that is new at level $\mathfrak{N}_p = \prod_{\mathfrak{P} \in S_K} \mathfrak{P}^{r'_{\mathfrak{P}}}$ such that
    \begin{align*}
    \bar{\rho}_{E,p} \sim \bar{\rho}_{\ff, \Bar{\omega}},
    \end{align*}

   Since we are in the case that $3|abc$, we can fix a prime ideal $\pp_3$ in $K$ above $3$ dividing $abc$. Then $\pp_3|\mathfrak{N}_{E_{(a,b,c)}}$. Let $f= f(\pp_3|3)$. Therefore, since we are in a level lowering case at $\pp_3$ we have that
   \begin{align}
   a_{\pp_3}(\ff) \equiv \pm (3^f + 1) \pmod{\Bar{\omega}}
   \end{align}
   since this is the necessary condition for Steinberg level raising. Here the congruence holds modulo the prime $\Bar{\omega}$ of the coefficient field $K_{\ff}$ above $p$ (and not necessarily modulo $p$): since $\pp_3 \nmid \mathfrak{N}_p$, the representation $\Bar{\rho}_{\ff,\Bar{\omega}}$ is unramified at $\pp_3$ with Frobenius trace $a_{\pp_3}(\ff) \bmod \Bar{\omega}$, while $E_{(a,b,c)}$ has multiplicative reduction at $\pp_3$.

   On the other hand, because of the Ramanujan--Petersson bound for Hilbert modular forms of parallel weight $2$ (proved by Blasius \cite{blasius2006hilbert}), we have that

   \begin{align}
   \label{ineq}
   |a_{\pp_3}(\ff)| \leq 2 \cdot 3^{f/2} .
   \end{align}
   And the same bound applies to all the Galois conjugates of $a_{\pp_3}(\ff)$. In particular, $a_{\pp_3}(\ff) \neq \pm (3^f + 1)$. Indeed, if $a_{\pp_3}(\ff) = \pm (3^f + 1)$ then $a_{\pp_3}(\ff)^2 = (3^f +1)^2 = (3^{f})^2 + 2 \cdot 3^f +1 > 4 \cdot 3^f +1$ which contradicts \eqref{ineq}. From this, we obtain a bound $F_K$ for $p$, since $\Bar{\omega}$ divides the non-zero algebraic integer $a_{\pp_3}(\ff) \mp (3^f + 1)$, so $p$ divides its norm. Observe that the bound depends only on $K$ because the form $\ff$ belongs to the space of Hilbert cuspforms over $K$ of parallel weight $2$ and level dividing a power of $2$, depending only on $K$. In particular, the degree of the algebraic integer $a_{\pp_3}(\ff)$ is bounded by the dimension of this space of forms, which depends only on $K$.\\

   Then, we arrive at a contradiction if $p> F_K$.

    

       \begin{proof}[Proof of Theorem \ref{3midabc}]

      It follows from the previous analysis and taking      \begin{equation*}
    C_K = \max\{D_K, A_K, B_K, F_K \}.
    \end{equation*}
\end{proof}

    \section{The case $3 \nmid xyz$} \label{3nmidabc}
    
    \subsection{Introduction to Pomey proof} \label{mimic_pomey}

    We want to analyze now what happens when $3\nmid xyz$. Massouti\'e \cite{massoutie1931dernier} and Pomey \cite{pomey1931nouvelles} proved in $1931$ the following result:
    
    \begin{thm}
    	\label{thm:pomey}
    If $p$ is an odd prime such that $p \equiv 2 \pmod 3$ and $x,y,z$ are non-zero integers such that $x^p+y^p+z^p=0$, then $3$ divides $x,y,$ or $z$.
    \end{thm}
    
    Observe that this result, in particular, proves that there is no non-trivial primitive solution to the equation $x^p + y^p +z^p =0$ regarding $3 \nmid xyz$ and $p$ odd such that $p \equiv 2 \pmod3$, which is the result we are looking for over $K$.
    \begin{itemize}
    	\item \textbf{Question: Does a similar result hold in other fields?} 
    \end{itemize}
    
    
    We mimic Pomey's proof of Theorem \ref{thm:pomey} to see which hypothesis we need to add to $K$ to obtain this result for $K$.
    
    Let $(x_1,x_2,x_3) \in \mathcal{O}_K^3$ be a non-trivial \emph{primitive} solution to the equation $x_1^p+x_2^p+x_3^p=0$, where $p$ is an odd prime with $p \equiv 2 \pmod 3$ (in particular $p \geq 5$). Primitivity will be used in an essential way in the arguments below (see Lemmas~\ref{lem:coprime6} and \ref{lem:split} and the proof of Theorem~\ref{pomey_case}), and the results of this section say nothing about non-primitive solutions. Let $\pp_3$ be a prime of $K$ above $3$ with $f(\pp_3|3)=1$ and $e(\pp_3|3)$ odd, which exists by condition (1) of Assumption~\ref{assumption}. Suppose $\pp_3 \nmid x_1 x_2 x_3$. Since $\mathcal{O}_K/\pp_3 = \FF_3$ (this is where we use that $f(\pp_3|3)=1$), we can write $x_i \equiv \varepsilon_i \pmod{\pp_3}$ with $\varepsilon_i = \pm 1$, and then $x_i^p \equiv \varepsilon_i \pmod{\pp_3}$ since $p$ is odd. Since $x_1^p + x_2^p+x_3^p=0$, we have that $\varepsilon_1=\varepsilon_2=\varepsilon_3$. Therefore $\pp_3|x_i^2-x_jx_k$ for $\{i,j,k\}=\{1,2,3\}$. Let us fix such $i,j,k$ and define 
    
    \begin{equation}
    	\label{eq:1}
    	P := \frac{x_i^{2p}-x_j^p x_k^p}{x_i^2 - x_j x_k} = \sum_{r=0}^{p-1} (x_i^2)^{p-1-r}(x_jx_k)^r \in \mathcal{O}_K.
    \end{equation}
    
    Observe that
    
    \begin{equation}
    \label{eq:Pformula}
    	P = p(x_i^2 x_j x_k)^{\frac{p-1}{2}} + m (x_i^2 - x_j x_k)^2
    \end{equation}
    for some $m \in \ZZ[x_i,x_j,x_k]$. Indeed, writing $U=x_i^2$ and $V=x_jx_k$, the polynomial $F(U,V)=\sum_{r=0}^{p-1}U^{p-1-r}V^r - p(UV)^{\frac{p-1}{2}} \in \ZZ[U,V]$ satisfies $F(U,U)=0$ and $\frac{\partial F}{\partial U}(U,U) = \left(\frac{p(p-1)}{2}-p\cdot\frac{p-1}{2}\right)U^{p-2}=0$, so $F$ is divisible by $(U-V)^2$, and since $(U-V)^2$ is monic as a polynomial in $U$, the quotient has integer coefficients.

    Since $\pp_3|x_i^2-x_j x_k$ and $x_i^2x_jx_k \equiv \varepsilon_i^2\varepsilon_j\varepsilon_k = 1 \pmod{\pp_3}$, 
    \begin{equation}
    	\label{eq:P}
    	P \equiv p \pmod{\pp_3}.
    \end{equation}
    Moreover, $P$ is totally positive: for every real embedding $\sigma$ of $K$, if $\sigma(U) \neq \sigma(V)$ then $\sigma(P) = \frac{\sigma(U)^p-\sigma(V)^p}{\sigma(U)-\sigma(V)}>0$ because $p$ is odd, and if $\sigma(U)=\sigma(V)$ then $\sigma(P)=p\,\sigma(U)^{p-1}>0$ since $U \neq 0$.
     
     On the other hand, writing $A = x_j^p+x_k^p$ and $B = x_j^p-x_k^p$ and using that $x_i^p = -A$, 
    
    \begin{equation}
    \label{eq:normform}
    	4 (x_i^{2p}-x_j^p x_k^p) = 3A^2 + B^2.
    \end{equation}
    
    Observe that the right hand side is $\norm_{L/K}(B+A\sqrt{-3})$, where $L = K(\sqrt{-3})$. In particular, $P$ divides $B^2+3A^2$.
    
    In Pomey's proof, he concluded that any positive \emph{odd} integer dividing $(x_i^{2p}-x_j^p x_k^p)$, and therefore $P$, will be of the form $a^2+3b^2$. This gives a contradiction since $p\equiv 2 \pmod 3$. Over a general field $K$ two difficulties appear. First, we have to control the primes above $2$ (and above $3$) dividing $P$; this is done in Lemma~\ref{lem:coprime6}. Second, the ring of integers $\mathcal{O}_L$ of $L=K(\sqrt{-3})$ is in general not known and it is in general strictly larger than $\mathcal{O}_K[\sqrt{-3}]$ (this happens already for $K=\QQ$), so a norm from $\mathcal{O}_L$ to $\mathcal{O}_K$ need not be of the form $a^2+3b^2$ with $a,b \in \mathcal{O}_K$. We overcome this difficulty by allowing $a,b \in K$ and by a valuation argument at $\pp_3$ (Lemma~\ref{lem:valuation}); this is where we use that $e(\pp_3|3)$ is odd.

   \subsection{Generalization of Pomey result} \label{general_of_pomey}
   
   Let $K$ be a totally real field and let $L = K(\sqrt{-3})$. We do not assume any knowledge of the ring of integers $\mathcal{O}_L$. We begin with two lemmas controlling the prime divisors of $P$.

   \begin{lemma}
   \label{lem:coprime6}
   Let $(x_1,x_2,x_3)$ be a non-trivial primitive solution to $x_1^p+x_2^p+x_3^p=0$ with $p$ an odd prime, $p \equiv 2 \pmod 3$, admissible for $K$, and let $P$ be as in \eqref{eq:1}. Then $P$ is coprime to $6$.
   \end{lemma}

   \begin{proof}
   Let $\qq$ be a prime of $K$ dividing $6$ and suppose that $\qq | P$. Then $\qq$ divides $x_i^{2p}-x_j^px_k^p$. If $\qq$ divided one of $x_i,x_j,x_k$, then this relation together with $x_1^p+x_2^p+x_3^p=0$ would imply that $\qq$ divides all of them, contradicting the primitivity of the solution. Hence $x_i^2$ and $x_jx_k$ are units in $\mathcal{O}_K/\qq$, whose multiplicative group is cyclic of order $\norm(\qq)-1$. Since $p \nmid \norm(\qq)-1$, the map $\xi \mapsto \xi^p$ is a bijection of $(\mathcal{O}_K/\qq)^\times$, and from $(x_i^2)^p \equiv (x_jx_k)^p \pmod{\qq}$ we deduce $x_i^2 \equiv x_jx_k \pmod{\qq}$. By \eqref{eq:Pformula}, $P \equiv p\,(x_i^2x_jx_k)^{\frac{p-1}{2}} \pmod{\qq}$, which is non-zero modulo $\qq$ because $p \geq 5$ is prime to $6$ and $\qq \nmid x_ix_jx_k$. This is a contradiction.
   \end{proof}

   \begin{lemma}
   \label{lem:split}
   Let $a,b,d \in \mathcal{O}_K$ with $d$ coprime to $6$, $d\mathcal{O}_K + a\mathcal{O}_K + b\mathcal{O}_K = \mathcal{O}_K$, and $d | a^2+3b^2$. Then every prime ideal $\qq$ of $K$ dividing $d$ splits in $L/K$.
   \end{lemma}

   \begin{proof}
   Let $\qq | d$. Then $\qq \nmid 6$. If $\qq | b$, then $\qq | a^2$, so $\qq | a$, contradicting $d\mathcal{O}_K + a\mathcal{O}_K + b\mathcal{O}_K = \mathcal{O}_K$. Hence $b$ is invertible modulo $\qq$ and $(ab^{-1})^2 \equiv -3 \pmod{\qq}$, i.e. the polynomial $X^2+3$ has two distinct roots modulo $\qq$ (they are distinct because $\qq \nmid 6$). Since $\qq$ does not divide the discriminant $-12$ of $X^2+3$, the Dedekind--Kummer theorem applies and $\qq$ splits in $L=K(\sqrt{-3})$.
   \end{proof}

   Observe that this is the point where it is essential to work with \emph{odd} divisors, as in Pomey's proof: for instance $2 | 1^2+3\cdot 1^2$, but $2$ does not split in $\QQ(\sqrt{-3})$.

   We will also need the following results.
   
   \begin{lemma} 
   	\label{thm:1}
   	Let $K$ be a totally real number field with an odd narrow class number. Then a unit $u$ is a square if and only if it is totally positive.
   \end{lemma}
   
   \begin{proof}
   	Let $C_K$ be the class group of $K$, $C_K^+$ the narrow class group of $K$, $U_K$ the group of units, $U_K^+$ the totally positive units, and $K^+$ the totally positive elements of $K$. Then we have the following exact sequence
    \begin{align*}
    1 \rightarrow U_K/U_K^+ \rightarrow K^{\times}/K^{+} \rightarrow C_K^{+} \rightarrow C_K \rightarrow 1
    \end{align*}
    Then,
    \begin{align*}
    |C_K^+| = h \frac{[K:K^+]}{[U_K:U_K^+]} = h \frac{2^r}{[U_K:U_K^+]}
    \end{align*}
    where $h$ is the class group number of $K$ and $r$ its number of real immersions. Since $|C_K^+|$ was odd, $[U_K:U_K^+]=2^r$ and therefore $|C_K^+|=h$. Then, $C_K = C_K^+$. Clearly, if $u$ is a square, then $u$ is totally positive. Suppose now that $u$ is totally positive, that is $u \in U^+$. By the Dirichlet Unit Theorem 
    \begin{align*}
    U = \{\pm 1\} \times \ZZ^{r+s-1}
    \end{align*}
    where $r$ and $s$ are the numbers of real and complex embeddings, respectively. Since $K$ is totally real, $s=0$ and $r=[K:\QQ]$. Then,
    \begin{align*}
    U_K/U_K^2 \cong \{\pm 1\} \times (\ZZ^{r-1})/(\ZZ^{r-1})^2.
    \end{align*}
    Therefore,
    \begin{align*}
    [U_K:U_K^2] = 2 \cdot 2^{r-1} = 2^r
    \end{align*}
    Then $2^r=[U_K:U_K^2]=[U_K:U_K^+][U_K^+:U_K^2]= 2^r[U_K^+:U_K^2]$, and we can conclude that $U_K^+ = U_K^2$
   \end{proof}

   \begin{lemma}
   	\label{class-field-1}
   	Let $K$ be a totally real field. Then $h(K)|h(K(\sqrt{-3}))$. Furthermore, if there exist a prime ideal in $K$ above $3$ ramifying in $K(\sqrt{-3})$ then $h^+(K)|h(K(\sqrt{-3}))$
   \end{lemma}
   \begin{proof}
   	Let $L=K(\sqrt{-3})$, which is a CM field. Let $H$ and $H^+$ be the Hilbert class field and the narrow Hilbert class field of $K$, so that $[H:K]=h(K)$ and $[H^+:K]=h^+(K)$. The extension $H/K$ is unramified at all places, so $H$ is totally real, whereas $L$ is totally imaginary; hence $L \cap H = K$. Then $LH/L$ is abelian of degree $h(K)$ and unramified at all places (at the infinite places because $L$ has no real places), so it is contained in the Hilbert class field of $L$ and $h(K) \mid h(L)$. For the second statement, $H^+/K$ is unramified at all finite primes, while $L/K$ is ramified at some finite prime by hypothesis; hence $L \cap H^+ = K$, and $LH^+/L$ is abelian of degree $h^+(K)$, unramified at all finite primes and therefore (since $L$ is totally imaginary) unramified everywhere. Thus $h^+(K) \mid h(L)$.
   \end{proof}
   
   \begin{remark}
   \label{ramification_3}
   	Let $\pp$ be a prime of $K$ above $3$ with $e(\pp|3)$ odd. Then $v_{\pp}(-3)=e(\pp|3)$ is odd, and the square root of an element of odd valuation generates a ramified quadratic extension of the completion $K_{\pp}$; hence $K_{\pp}(\sqrt{-3})/K_{\pp}$ is ramified. Hence, to guarantee having a prime ideal in $K$ above $3$ ramifying in $K(\sqrt{-3})$, it is enough to have a prime ideal above $3$ of odd ramification index. Here is where we need that the prime $\pp_3$ of condition (1) of Assumption~\ref{assumption} has odd ramification index. (Conversely, since $3$ is odd, a prime above $3$ with even ramification index is unramified in $K(\sqrt{-3})$.)
   \end{remark}

   \begin{lemma}
   \label{lem:valuation}
   Let $\pp_3$ be a prime of $K$ above $3$ with $e(\pp_3|3)$ odd. Let $x,y \in K$ (not necessarily integral) be such that $\gamma := x^2+3y^2 \in \mathcal{O}_K$ and $\pp_3 \nmid \gamma$. Then $x$ is a $\pp_3$-adic unit and $v_{\pp_3}(3y^2)>0$. In particular, the reduction of $\gamma$ modulo $\pp_3$ is a non-zero square in $\mathcal{O}_K/\pp_3$.
   \end{lemma}

   \begin{proof}
   Write $v = v_{\pp_3}$. If $y = 0$, then $v(x^2)=v(\gamma)=0$ and the statement is clear. Otherwise $v(x^2)=2v(x)$ is even (or $x=0$), while $v(3y^2)=e(\pp_3|3)+2v(y)$ is odd. Hence $v(x^2)\neq v(3y^2)$ and therefore $0 = v(\gamma) = \min\{v(x^2),v(3y^2)\}$. Since $v(3y^2)$ is odd, it is not $0$, so $v(x^2)=0$ and $v(3y^2)>0$. Then, in the valuation ring of $\pp_3$, we have $\gamma \equiv x^2 \pmod{\pp_3}$ with $x$ a $\pp_3$-adic unit, and the residue field of this valuation ring is $\mathcal{O}_K/\pp_3$.
   \end{proof}

   Then we can prove the results stated in the introduction.
  
   \begin{proof}[Proof of Theorem \ref{divpomey}]
   Let $t = h(L)$, which is odd by Assumption~\ref{assumption}, and let $c$ denote the non-trivial element of $\Gal(L/K)$ (complex conjugation). By Lemma~\ref{lem:split}, we can write
   \begin{align*}
   d\,\mathcal{O}_K = \prod_{i=1}^{r} \qq_i^{a_i}, \qquad \qq_i \mathcal{O}_L = \id{Q}_i \, c(\id{Q}_i),
   \end{align*}
   with $\id{Q}_i \neq c(\id{Q}_i)$ prime ideals of $\mathcal{O}_L$. Since $t$ is the class number of $L$, there exist $\gamma_i \in \mathcal{O}_L$ with $\id{Q}_i^t = \gamma_i \mathcal{O}_L$, and then $c(\id{Q}_i)^t = c(\gamma_i)\mathcal{O}_L$. Let $\beta_0 = \prod_{i=1}^r \gamma_i^{a_i} \in \mathcal{O}_L$. Then
   \begin{align*}
   d^t \mathcal{O}_L = \prod_{i=1}^r \left(\id{Q}_i\, c(\id{Q}_i)\right)^{t a_i} = \beta_0\, c(\beta_0)\, \mathcal{O}_L = \norm_{L/K}(\beta_0)\,\mathcal{O}_L.
   \end{align*}
   Hence $d^t = u \cdot \norm_{L/K}(\beta_0)$ with $u \in \mathcal{O}_L^\times$. Since $d^t$ and $\norm_{L/K}(\beta_0)$ belong to $K$, $u$ is fixed by $c$, so $u \in \mathcal{O}_L^{\times}\cap K = \mathcal{O}_K^\times$. Moreover, $d$ is totally positive and $\norm_{L/K}(\beta_0)$ is totally positive (as $L$ is a CM field, for every real embedding of $K$ it is the square of the absolute value of a complex number), so $u$ is totally positive.

   Now, the prime $\pp_3$ of condition (1) of Assumption~\ref{assumption} has odd ramification index, so it ramifies in $L$ by Remark~\ref{ramification_3}. Then $h^+(K)$ divides $t$ by Lemma \ref{class-field-1}, so $h^+(K)$ is odd. Therefore, the unit $u$ is a square by Lemma \ref{thm:1}, say $u = v^2$ with $v \in \mathcal{O}_K^\times$. Then
   \begin{align*}
   d^t = v^2\,\beta_0\,c(\beta_0) = \norm_{L/K}(\beta), \qquad \beta := v\beta_0 \in \mathcal{O}_L.
   \end{align*}
   Finally, since $L = K(\sqrt{-3})$, we can write $\beta = x + y\sqrt{-3}$ with $x,y \in K$, and then $d^t = x^2+3y^2$. Note that, since $\mathcal{O}_L$ may be strictly larger than $\mathcal{O}_K[\sqrt{-3}]$, $x$ and $y$ need not be integral.
   \end{proof}
   
      \begin{proof}[Proof of Theorem \ref{pomey_case}]
   Let $(x_1,x_2,x_3)$ be a non-trivial primitive solution and suppose that $\pp_3 \nmid x_1x_2x_3$. Let $P$, $A$ and $B$ be as in subsection \ref{mimic_pomey}. Then $P$ is totally positive, $P \equiv p \pmod{\pp_3}$ by \eqref{eq:P}, and $P$ divides $B^2+3A^2$ by \eqref{eq:normform}. By Lemma~\ref{lem:coprime6}, $P$ is coprime to $6$. Moreover, $P\mathcal{O}_K + A\mathcal{O}_K + B\mathcal{O}_K = \mathcal{O}_K$: indeed, a prime $\qq$ dividing $P$, $A$ and $B$ does not divide $2$ and divides $A+B = 2x_j^p$ and $A-B=2x_k^p$, so it divides $x_j$ and $x_k$, and hence also $x_i$, contradicting primitivity. Therefore we can apply Theorem~\ref{divpomey} with $d=P$, $a=B$ and $b=A$, and we obtain $P^t = x^2+3y^2$ for some $x,y \in K$, where $t = h(L)$ is odd. Since $P^t \in \mathcal{O}_K$ and $\pp_3 \nmid P^t$, Lemma~\ref{lem:valuation} shows that $P^t$ is a non-zero square modulo $\pp_3$. Since $\mathcal{O}_K/\pp_3 = \FF_3$, this means $P^t \equiv 1 \pmod{\pp_3}$. On the other hand, by \eqref{eq:P}, $P^t \equiv p^t \equiv (-1)^t = -1 \pmod{\pp_3}$, because $p \equiv 2 \pmod 3$ and $t$ is odd. This is a contradiction, and therefore $\pp_3 | x_1x_2x_3$.
\end{proof}

  Finally, from Theorem \ref{pomey_case} and Theorem \ref{3midabc} we conclude Theorem \ref{thm:main}.

  \begin{proof}[Proof of Theorem \ref{thm:main}]
  Let $C_K'$ be the bound of Theorem~\ref{3midabc} and let $C_K = \max\{C_K', \norm(\qq) : \qq \mid 6\}$. Let $p > C_K$ with $p \equiv 2 \pmod 3$; in particular $p$ is admissible for $K$. Let $(x,y,z)$ be a non-trivial primitive solution to $x^p+y^p+z^p=0$. Let $\pp_3$ be a prime of $K$ as in condition (1) of Assumption~\ref{assumption}. By Theorem~\ref{pomey_case}, $\pp_3$ divides $xyz$, so $3|xyz$ in the sense of Section~\ref{3|abc}, which contradicts Theorem~\ref{3midabc}.
  \end{proof}
   
   

\section{Examples of fields K} \label{ejemplos}
   
We want to give an example of totally real fields $K$ satisfying Assumption~\ref{assumption}. A natural question is whether there exists a quadratic extension satisfying the assumption. 
    
    \begin{prop}
    	\label{quadratic}
     Let $p$ be an odd prime. Then there are no quadratic totally real extensions $K$ of $\QQ$ such that $h(K(\sqrt{-p}))$ is odd and such that $p$ splits in $K$.
    \end{prop}
    
    \begin{proof}
    Suppose there exist such a quadratic extension $K:=\QQ(\sqrt{d})$ and that $\pp$ splits as $\pp_1 \pp_2$ in $K$. Then, $K(\sqrt{-p})$ ramifies at $\pp_1$ and $\pp_2$ and probably at the primes above $2$ in $K$. By Lemma \ref{class-field-1}, we also have that $h(K)$ is odd. Let $\mathcal{A}_{K(\sqrt{-p})/K}$ be the number of ambiguous classes of $K(\sqrt{-p})/K$ which are in the $2$-class group $K(\sqrt{-p})$. Then using the \textit{ambiguous class number formula} of Chevalley (see \cite[Chapter 13, Lemma 4.1]{lang2012cyclotomic} and \cite[Section 3]{mccall1995imaginary}) we obtain that
    \begin{align*}
    	\mathcal{A}_{K(\sqrt{-p})/K} = \frac{2^{1+\gamma} }{[\mathcal{O}_K^{\times}:\mathcal{O}_K^{\times} \cap \text{Norm}(K(\sqrt{-p})^{\times})]}
    \end{align*}
    where $\gamma$ is the number of prime ideals of $K$ ramifying in $K(\sqrt{-p})$. Since $K(\sqrt{-p})$ ramifies at $\pp_1$ and $\pp_2$, we obtain that $\gamma \geq 2$. Since $K$ is totally real and $-1$ is not a norm we have that $[\mathcal{O}_K^{\times}:\mathcal{O}_K^{\times} \cap  \text{Norm}(K(\sqrt{-p})^{\times})]= 2^e$ with with $1 \leq e \leq 2$. Then, $2|\mathcal{A}_{K(\sqrt{-p})/K}$. However, this is a contradiction since $\text{Cl}^{(2)}(K(\sqrt{-p})) = \{e\}$ because $h(K(\sqrt{-p}))$ was odd. Then we conclude that there not exist a quadratic extension of $\QQ$ satisfying what we wanted.
    \end{proof}
    
    Therefore, it implies in particular the following Corollary.
    \begin{coro}
    \label{coro:quadratic}
    There are no quadratic totally real extensions of $\QQ$ satisfying  Assumption~\ref{assumption}.
    \end{coro}
    \begin{proof}
    This follows from Proposition \ref{quadratic} and the fact that satisfying Assumption~\ref{assumption} in a quadratic field is equivalent to asking that $3$ splits.
    \end{proof}

For a general totally real field we have the following result. 
        
    	\begin{prop}
    		Let $p$ be an odd prime. Let $K$ be a totally real field with $h(K(\sqrt{-p}))$ odd and such that $V_p \neq \emptyset$. 
            Then $V_p = \{\pp\}$ with $\pp$ a prime ideal of $K$ above $p$. In particular,  the only odd prime ideal ramifying in $K(\sqrt{-p})/K$ is $\pp$. 
    	\end{prop}
    	
    	\begin{proof}
                Suppose there exists a prime ideal $\pp' \in V_p$ different from $\pp$. In particular, it ramifies in $K(\sqrt{-p})/K$. Then, using the Chevalley formula for ambiguous class numbers, we obtain that
                  \begin{equation}
                  \label{eq:ambiguous}
    	    \mathcal{A}_{K(\sqrt{-p})/K} = 
                 \frac{2^{\gamma + r + s -1} }{[\mathcal{O}_K^{\times}:\mathcal{O}_K^{\times} \cap \text{Norm}(K(\sqrt{-p})^{\times})]}
                 \end{equation}
                 where $\gamma$ is the number of finite prime ideals of $K$ ramifying in $K(\sqrt{-p})$, and $r$ and $s$ are the real and complex embeddings of $K$ respectively. Since $K$ is totally real, the primitive roots of unity in there are $\{\pm 1\}$. Then, using Dirichlet's unit theorem we obtain that $[\mathcal{O}_K^{\times}:\mathcal{O}_K^{\times} \cap \text{Norm}(K(\sqrt{-p})^{\times})] = 2^k$, with $0\leq k \leq r+s$. Therefore, from \eqref{eq:ambiguous} we have that $\mathcal{A}_{K(\sqrt{-p})/K} = 2^e$ with $\gamma - 1 \leq e$. Since we have, at least two prime ideals $\pp$ and $\pp'$ of $K$ ramifying in $K(\sqrt{-p})$ we deduce that $\gamma \geq 2$. Therefore $\gamma -1 \geq 1$ and thus $2|\mathcal{A}_{K(\sqrt{-p})/K}$, arriving to a contradiction since $h(K(\sqrt{-p}))$ was odd.
    	\end{proof}

\begin{remark}
\label{options}
If a totally real field satisfies the hypothesis of Theorem \ref{thm:main}, it has to satisfy then that $V_3=\{\pp_3\}$. The  possibilities for the ramification of $3$ are the following: 
        \begin{enumerate}
        \item $3$ totally ramifies in $K$ and $[K:\QQ]$ is odd,
        \item $3$ splits as $\pp_0^k \cdot \prod_{i=1}^n \pp_i^{2k_i}$ where $k$ is odd and $f(\pp_0|3)=1$.
        \end{enumerate}
        Indeed, by condition (1) of Assumption~\ref{assumption} the unique prime $\pp_3$ of $V_3$ must have residue degree $1$; in case (1) this is automatic.
        
        In particular, in the second case $K$ cannot be a Galois extension.
\end{remark}

\begin{remark}
\label{rem:odd_degree}
If $K$ satisfies Assumption~\ref{assumption}, then $[K:\QQ]$ is odd. Indeed, by the previous discussion we can write $3\mathcal{O}_K = \pp_0^k \cdot \prod_{i=1}^n \pp_i^{2k_i}$ with $k$ odd and $f(\pp_0|3)=1$, so that $[K:\QQ] = k + \sum_{i=1}^n 2k_i f(\pp_i|3)$ is odd. In particular, this gives another proof of Corollary~\ref{coro:quadratic}, and it explains why only fields of odd degree appear in Table~\ref{table:1} (totally real quintic fields have absolute discriminant larger than $2000$).
\end{remark}

     \subsection{Cubic extensions}   

     In \cite{alvarado2021robust}, while developing an algorithm to compute the $S$-unit solutions, they listed in Table 8.1 the totally real cubic extensions $K$ in which $2$ totally ramifies and $|\Delta_K| \leq 2000$. Then they verify that the $S$-unit solutions satisfy either condition (A) or (B) of \cite[Theorem 3]{freitas2015asymptotic}.

     In our case, we do not have to deal with $S$-unit equation, so the computation is straightforward since we only have to compute $h(K(\sqrt{-3}))$ and $3 \cdot \mathcal{O}_K$. Let $\mathscr{K}_X$ denote the set of totally real number fields satisfying Assumption~\ref{assumption} with $h(K(\sqrt{-3}))$ odd and which have absolute discriminant $\Delta_K$ satisfying $|\Delta_K|\leq X$. Table \ref{table:1} lists all the fields of $\mathscr{K}_X$ for $X=2000$. In all the fields of Table~\ref{table:1} the prime above $3$ with odd ramification index has residue degree $1$ (the factorizations are $\pp^3$ and $\pp_1\pp_2^2$, and $[K:\QQ]=3$), so they satisfy condition (1) of Assumption~\ref{assumption} in its present form. In particular, it follows from Theorem \ref{thm:main} that for each of these cubic fields $K$ there exists a bound $C_K$ such that, for every prime $p \equiv 2 \pmod{3}$ with $p > C_K$, there are no non-trivial primitive solutions in $\mathcal{O}_K$ to $x^p+y^p+z^p=0$.

     \begin{table}[h!]
    \centering
 \begin{tabular}{llll}
        \toprule
        $f_K$ & $|\Delta_K|$ & $h(K(\sqrt{-3})$ & Ramification of $3$\\
        \midrule
        $x^3-3x-1$ & $81$ & $1$ & $\pp^3$\\
        $x^3-x^2-4x+1$ & $321$ & $1$ & $\pp_1 \pp_2^2$\\
        $x^3-x^2-5x+3$ & $564$ & $1$ & $\pp_1 \pp_2^2$\\
        $x^3-6x-3$ & $621$ & $1$ & $\pp^3$\\
        $x^3-6x-2$ & $756$ & $1$ & $\pp^3$\\
        $x^3-6x-1$ & $837$ & $1$ & $\pp^3$\\
        $x^3-x^2-6x+2$ & $993$ & $1$ & $\pp_1 \pp_2^2$\\
        $x^3-x^2-9x+12$ & $1101$ & $1$ & $\pp_1 \pp_2^2$\\
        $x^3-x^2-8x-3$ & $1425$ & $1$ & $\pp_1 \pp_2^2$\\
        $x^3-x^2-7x+1$ & $1524$ & $1$ & $\pp_1 \pp_2^2$\\
        $x^3-12x-14$ & $1620$ & $1$ & $\pp^3$\\
        $x^3-9x-6$ & $1944$ & $1$ & $\pp^3$\\
        \bottomrule
        \end{tabular}
\caption{Fields in $\mathscr{K}_{2000}$ and how $3$ splits.}
\label{table:1}
\end{table}

\begin{remark}
Observe that $3$ splits in the way explained in Remark \ref{options}. Case $1$ of the Table \ref{table:1} corresponds to the totally real subfield of $\QQ(\zeta_9)$. In the subsection \ref{cyclotomic}, we deal deeply with these cases.
\end{remark}

 \subsection{Cyclotomic extensions} \label{cyclotomic}

Let $L$ be a cyclotomic extension. Suppose $L=\QQ(\zeta_{p^k})$ for $p$ an odd prime number and $k$ a positive integer. We first would like to find examples that satisfy the hypothesis about the splitting of $3$ of Theorem \ref{thm:main}. Let $K$ be the totally real subfield of $L$. Since $K$ is a Galois extension, $3$ has to ramify in $K$ as it is explained in \ref{options}. Then, the only examples that may satisfy the desired hypothesis are the fields $\QQ(\zeta_{3^n})^+$.The following result guarantees us the odd parity of the class number of these fields.







\begin{prop}
\label{prop-parity}
(\cite[Proposition 3]{HORIE_2002}) Let $\ell$ be a prime congruent to $2$ or $5$ modulo $9$. Then, for any positive integer $n$, the class number of the cyclotomic field $\QQ(\zeta_{3^n})$ is relatively prime to $\ell$.
\end{prop}

Then, using this latter result we are able to prove that the main theorem \ref{thm:main} works on $\QQ(\zeta_{3^n})^+$ as it is stated in \ref{cubic_cyclo}.

\begin{proof}[Proof of Theorem \ref{cubic_cyclo}]
Let us denote $K=\QQ(\zeta_{3^n})^+$. Observe that $K(\sqrt{-3})= \QQ(\zeta_{3^n})$, then $h(K(\sqrt{-3}))$ is odd by Proposition \ref{prop-parity}. Since $[K:\QQ]=3^{n-1}$ we have that $S_3=T_3=V_3=\{\pp_3\}$ with $\pp_3$ the only prime above $3$ in $K$; in particular $T_3 \cap V_3 \neq \emptyset$. Then the result follows from Theorem \ref{thm:main}.
\end{proof}


In \cite[Theorem 4]{freitas2020asymptotic}, Freitas, Kraus, and Siksek proved that under the condition of $2$ being inert in the maximal subfield of $\QQ(\zeta_{q^r})^+$ of $q$-power degree, the asymptotic Fermat's Last Theorem holds for such a field as long as $q\geq 5$. As it is stated in \cite[Remark 3.3]{freitas2020asymptotic}, the proof of Theorem 4 fails for $q=3$ since there are unit solutions to the equation 
\begin{equation}
\label{eq:S-unit}
\lambda + \mu = 1, \ \ \lambda, \mu \in \mathcal{O}_K^+.
\end{equation}

However, for our result, we do not have to deal with the $S$-unit equation, which allows us to cover the case $q=3$, at least for half of the prime exponents. Observe that another important difference with the case $q > 3$ is that we obtain a result for $q = 3$ over the full extension $\QQ(\zeta_{3^r})^+$. In \cite[Theorem 1.2]{chen2024fermat}, Chen et al. proved a similar result for the field $\QQ(\sqrt{5})$. They established Fermat's Last Theorem over $\QQ(\sqrt{5}) = \QQ(\zeta_5)^+$ for all prime exponents $p$ satisfying $p \equiv 5, 7 \pmod{8}$ or $p \equiv 19, 41 \pmod{48}$. Note that this is not an asymptotic result, since it holds for all such primes $p$. Aside from this result, for $q > 3$, no results on Fermat's Last Theorem are known over the full extensions $\QQ(\zeta_{q^r})^+$; the known results hold only over proper subfields of these fields.

Finally, let us observe the following:

\begin{prop}
$2$ is inert in $\QQ(\zeta_{3^n})^+$ for all $n$.
\end{prop}
\begin{proof}
Follows directly from \cite[Lemma 2.1]{freitas2020asymptotic}
\end{proof}

Then, the fields $\QQ(\zeta_{3^r})^+$ are examples of fields over which there are no non-trivial primitive solutions to the Fermat equation for all sufficiently large primes $p$ such that $p \equiv 2 \pmod{3}$, but do not satisfy neither condition (A) nor (B) of \cite[Theorem 3]{freitas2015asymptotic}. Indeed, let $\mathfrak{P}$ be the prime ideal above $2$. Then, (A) does not hold since in this case $T=\emptyset$ and (B) does not hold since there are unit solutions $(\lambda,\mu)$ to the unit Equation (\ref{eq:S-unit}) (see \cite[Lemma 16]{siksek2024curves}). Then $v_{\mathfrak{P}}(\lambda \mu) = 0 \not\equiv 1 \pmod{3}$ where $S=\{\mathfrak{P}\}$.

\bibliographystyle{plain}
\bibliography{biblioppp}
\end{document}